\documentclass[12pt]{amsart}
\usepackage{natbib}
\usepackage{booktabs}
\usepackage{latexsym}
\usepackage{amsthm}
\usepackage{amsmath}
\usepackage[dvips]{graphics}
\usepackage{graphicx}
\usepackage{amssymb}
\usepackage{color}

\begin{document}

\newcommand{\tabsize}{\fontsize{8.8}{10.5pt}\selectfont}

\newtheorem*{theorem*}{Theorem}
\newtheorem*{lemma*}{Lemma}
\newtheorem*{claim*}{Claim}
\newtheorem*{exercise*}{Exercise}
\newtheorem*{note*}{Note}
\newtheorem*{example*}{Example}
\newtheorem*{problem*}{Problem}
\newtheorem*{solution*}{Solution}
\newtheorem*{remark*}{Remark}
\newtheorem{theorem}{Theorem}[section]
\newtheorem{proposition}[theorem]{Proposition}
\newtheorem{remark}[theorem]{Remark}
\newtheorem{conjecture}[theorem]{Conjecture}
\newcommand{\x}{\text{\boldmath{$X$}}}
\newcommand{\ssum}[2] {\overset{#2}{\underset{#1}{\sum}}}  
\newcommand{\nature}{\ensuremath{\mathbb{N}}}   
\newcommand{\integer}{\ensuremath{\mathbb{Z}}}  
\newcommand{\rational}{\ensuremath{\mathbb{Q}}} 
\numberwithin{equation}{section}

\title[Representing Sets with Triangular Numbers]{Representing Sets with Sums of Triangular Numbers}
\author{Ben Kane}
\address{Department of Mathematics, Radboud Universiteit, Toernooiveld 1, 6525 AJ, Nijmegen, Netherlands}
\email{bkane@science.ru.nl}
\date{\today}
\begin{abstract}
We investigate here sums of triangular numbers $f(x):=\ssum{i}{} b_i T_{x_i}$ where $T_n$ is the $n$-th triangular number.  We show that for a set of positive integers $S$ there is a finite subset $S_0$ such that $f$ represents $S$ if and only if $f$ represents $S_0$.  However, computationally determining $S_0$ is ineffective for many choices of $S$.  We give an explicit and efficient algorithm to determine the set $S_0$ under certain Generalized Riemann Hypotheses, and implement the algorithm to determine $S_0$ when $S$ is the set of all odd integers. 
\end{abstract}
\maketitle

\section{Introduction}\label{introsection}

In 1638 Fermat wrote that every number is a sum of at most three triangular numbers, four square numbers, and in general $n$ polygonal numbers of order $n$.  Here the triangular numbers are $T_x:=\frac{x(x+1)}{2}$, where we include $x=0$ for simplicity.  The claim for four squares was shown by Lagrange in 1772, while Gauss famously wrote ``Eureka, $\triangle+\triangle+\triangle=n$'' in his mathematical diary on July 10, 1796.  
\begin{theorem*}[Gauss, 1796]
Every positive integer is the sum of three triangular numbers.
\end{theorem*}
The first proof of the full assertion of Fermat was given by Cauchy in 1813.  

In 1917, Ramanujan extended the question about four squares to consider which choices of $b=(b_1,b_2,b_3,b_4)$ satisfy $b_1x_1^2+b_2x_2^2+b_3x_3^2+b_4x_4^2$ representing every positive integer.  We shall refer to such forms as \begin{it}universal diagonal forms\end{it}.  He gives a list of 55 possible choices of $b$ which he then claims are the complete list of universal quarternary diagonal forms (54 forms actually turned out to be universal).  

In 1862 Liouville similarly proved the following generalization of Gauss's theorem.
\begin{theorem*}
Let $a,b,c$ be positive integers with $a\leq b\leq c$.  Then every $n\in\nature$ can be written as $aT_{x}+b T_{y}+cT_{z}$ if and only if $(a,b,c)$ is one of the following: 
$$
(1,1,1),\, (1,1,2),\, (1,1,4),\, (1,1,5),\, (1,2,2),\, (1,2,3),\, (1,2,4).
$$
\end{theorem*}
In fact, the following simple condition determines whether a fixed set of positive integers $(b_1,\dots,b_k)$ give rise to a sum of triangular numbers $\sum_{i=1}^k b_k T_{x_k}$ which represents every integer as shown by the author in \cite{Kane3}.
\begin{theorem}\label{thm8}
Fix the sequence $b_1\leq \dots\leq b_k\in\nature$.  Then
\newcounter{Lcount}
\begin{list}{\arabic{Lcount}.}{\usecounter{Lcount}}
\item The sum of triangular numbers
$$
f(x):=f_b(x):=\ssum{i=1}{k} b_iT_{x_i}
$$  
represents every positive integer if and only if $f_b$ represents the integers $1$, $2$, $4$, $5$, and $8$.
\item  The corresponding diagonal quadratic form $Q(x)=\ssum{i=1}{k}b_ix_i^2$ with $x_i$ all odd represents every integer of the form $8n+\ssum{i=1}{k}b_i$
with $n\geq 0$ if and only if it represents $8+\ssum{i=1}{k}b_i$, $16+\ssum{i=1}{k}b_i$, $32+\ssum{i=1}{k}b_i$, $40+\ssum{i=1}{k}b_i$, and $64+\ssum{i=1}{k}b_i$. 
\end{list}
\end{theorem}

Recently, Conway and Schneeberger proved a very nice similar condition for positive definite quadratic forms whose corresponding matrix has integer entries, but without publishing their results.  
\begin{theorem*}[Conway-Schneeberger]
A positive definite quadratic form $Q(x)=x^tAx$ where $A$ is a positive symmetric matrix with integer coefficients represents every positive integer if and only if it represents the integers $1,2,3,5,6, 7, 10, 14,$ and $15$.  
\end{theorem*}

Bhargava gave an elegant simpler proof of the Conway-Schneeberger 15 theorem in \cite{Bhargava1}, in addition to showing more generally that for any set $S\subseteq \nature$ it is always sufficient to check whether $Q$ represents a finite subset $S_0$, and showed the set $S_0$ for the two sets $S=\{ 2n+1: n \in \integer^{+}\}$ and $S=\{ p\text{ prime}\}$.

In this paper, we will consider a similar generalization of Theorem \ref{thm8}.
\begin{theorem}\label{SThm}
Let a set $S\subseteq \nature$ be given.  Then there is a finite subset $S_0$ of $S$ such that $f(x)$ represents $S$ if and only if $f$ represents $S_0$.
\end{theorem}

A simple computer calculation leads us to conjecture a set $S_0$ when $S$ is the set of all odd integers, for example.
\begin{conjecture}\label{OddConjecture}
A sum of triangular numbers $f$ represents all odd integers if and only if it represents the integers
$$
1,5,7,9,11,13,17,19, 25, 29, 35,49,89.
$$
\end{conjecture}

Unlike in Bhargava's theorem however, current techniques are insufficient for computationally determining a suitable $S_0$ for most choices of $S$, due to ineffective bounds for the class numbers of imaginary quadratic fields.  We shall briefly explain this complication.  Let $f(x)=b_1T_{x_1}+b_2T_{x_2}+b_3T_{x_3}$ be given such that $f$ represents all of the integers in $S$, but the corresponding (diagonal) quadratic form is not (spinor) genus $1$.  Then the corresponding weight $3/2$ modular form with $x_i$ all odd can be written as an Eisenstein series plus a cusp form.  Siegel has shown that the Fourier coefficients (with bounded divisibility at the anisotropic primes) of the Eisenstein series grow like the class number \cite{Siegel2}.  Siegel has also shown that the class number grows faster than $n^{\frac{1}{2}-\epsilon}$ \cite{Siegel1}, but the bound was ineffective because Siegel showed the result by first assuming a certain Riemann hypothesis and showing the result, and then assuming the Riemann hypothesis was false, and getting a different constant, depending on the location of possible zeros.  The best known effective bound is given by Oesterle \cite{Oesterle1}, but is only $O(\log n)$.  After decomposing the cusp form into $g_1+g_2$ where $g_1$ is in the space of lifts of one dimensional theta series and the Shimura lift of $g_2$ is cuspidal, we note that the coefficients of $g_2$ grow slower than $n^{\frac{1}{2}-\frac{1}{28}+\epsilon}$ by the work of Duke \cite{Duke1}, and $g_1$ is supported at finitely many square classes with the same growth as the Eisenstein series.  Since the coefficients with bounded divisibility at the anisotropic primes of the Eisenstein series grow faster than the coefficients of $g_2$, every sufficiently large $n$ with bounded divisibility by the anisotropic primes and outside of the support of the coefficients of $g_1$ must be represented.  However, we do not know the implied constant from Siegel's ineffective bound, so we cannot effectively determine when $n$ is sufficiently large.

Assuming the Generalized Riemann Hypothesis for Dirichlet $L$-functions, and using Duke's (effective) bound of $n^{\frac{3}{7}+\epsilon}$ \cite{Duke1}, we would have an algorithm to determine whether $f$ represents $s$ or not.  However, although Duke's result is effective, the author is unaware of any paper where the constant is explicitly computed.  Even assuming that the implied constant was $1$, the bound obtained is entirely infeasible with current computer technology.  Using an idea of Ono and Soundararajan \cite{OnoSound1}, and a generalization by the author \cite{Kane1}, we will be able to determine an algorithm to determine the set $S_0$ under the additional Generalized Riemann Hypothesis for $L$-functions of weight $2$ newforms.  For notational ease, we will refer to an integer which is locally represented at every prime as an eligible integer.

\begin{theorem*}[Ono-Soundararajan \cite{OnoSound1}]
Assume GRH for Dirichlet $L$-functions and GRH for $L$-functions of weight $2$ newforms.  Then the eligible integers not represented by Ramanujan's ternary quadratic form $x^2+y^2+10z^2$ are precisely
$$
3,7,21, 31,33,43, 67, 79, 87, 133,217,219,223,253,307,391,679,2719.
$$
\end{theorem*}

Using the ideas in \cite{OnoSound1}, \cite{Kane1}, and \cite{Kane2}, we obtain the following result.
\begin{theorem}\label{DetermineS0Thm}
Assume GRH for Dirichlet $L$-functions and GRH for $L$-functions of weight $2$ newforms.  Let a set of positive integers $S$ of nonzero density be given.  Then there exists an explicit algorithm to determine the (unique) smallest (finite) set $S_0$ such that a sum of triangular numbers $f$ represents $S$ if and only if $f$ represents $S_0$.

Moreover, this algorithm is computationally feasible for sets $S$ which contain many small integers.
\end{theorem}
We also implement the algorithm to prove Conjecture \ref{OddConjecture} under these GRH assumptions.
\begin{theorem}\label{OddThm}
Assume GRH for Dirichlet $L$-functions and GRH for $L$-functions of weight $2$ newforms.  Then a sum of triangular numbers $f$ represents all odd integers if and only if it represents the integers
$$
1,5,7,9,11,13,17,19, 25, 29, 35,49,89.
$$
Moreover, this is the smallest such set with this property.
\end{theorem}
\begin{remark}
We only need to assume GRH specifically for weight $2$ newforms that occur in the decomposition of any $f$ for which, under the GRH assumption we can show will generate all integers in $S$.  However, to make our algorithm more efficient, we will make the GRH assumption in general so that we do not need to check very many coefficients of a large number of weight 2 modular forms (For more information on modular forms, please see \cite{Ono1}).  Thus, we note here that for determining the set $S_0$ for $S$ all odd integers, we could have avoided the GRH assumptions except in 3 cases at the added disadvantage of more computation.
\end{remark}

We conclude with a curious example.  Consider the set 
$$
S:=\{ n\in\nature: 2x^2+3y^2+4z^2 = 8n+9 \text{ has a solution}\}.
$$
Current techniques appear insufficient to determine the set $S$.  However, we are able to prove the following.
\begin{theorem}\label{curiousthm}
The sum of triangular numbers $f$ represents $S$ if and only if $f$ represents
$$
2,3,4,5, 10, 16, 17, 19, 89.
$$
\end{theorem}
Notice that the above theorem is unconditional.  As we will see in Section \ref{DetermineSection}, GRH predicts that 
$$
S=\widetilde{S}:=\nature\setminus \{ 1,8,31 \}.
$$
However, without GRH we know of no algorithm to determine the set $\widetilde{S}_0$.

\section{Existence of $S_0$}\label{S0Section}

We will begin by showing that fixing a subset $S$ of the positive integers, that indeed checking a finite subset $S_0$ will suffice.

\begin{proof}[Proof of Theorem \ref{SThm}]
We will follow the basic argument of escalator lattices in Bhargava
's argument \cite{Bhargava1} for quadratic forms.  Without loss of generality, we will denote the triangular form $f=f_b$ simply by the sequence of coefficients, $b=[b_1,\dots,b_k]$, with $b_1\leq b_2\leq\dots\leq b_k$.  Note that we must represent the smallest integer $s_{\emptyset}\in S$, so it follows easily that $b_1\leq s_{\emptyset}$.  For each choice of $b_1$, we find the smallest integer $s_{[b_1]}\in S$ which is not represented by $[b_1]$, and conclude that $b_2\leq s_{[b_1]}$.  We recursively continue to build a tree of possible choices of $b_k$ depending on the previous choices of $b_1,\dots, b_{k-1}$.  Note that we never need the to choose the same integer $b_i$ more than 3 times in one branch, since this then precisely represents every integer congruent to zero modulo $b_i$ by Gauss's theorem.  Whenever $b$ represents all odd integers, we will say that $b$ is a leaf of the tree, since any arbitrary choice $b'$ containing $b$ as a subsequence will automatically represent every integer in $S$.  For $b\in T$ not a leaf, we will denote by $s_b$ the smallest $s\in S$ not represented by $b$, and we will call $s_b$ the \begin{it}truant\end{it} of $b$.  Thus, taking for our tree $T$
$$
S_0:=\{ s_b: b \in T\}
$$
it follows easily that a triangular form $f$ represents $S$ if and only if it represents $S_0$ by our construction of $T$.  Moreover, such a choice of $S_0$ is smallest possible and unique, since for the form $b$, $s_b$ is the smallest $s\in S$ not represented by $b$, noting that 
$$
[b_1,\dots,b_k, s_b+1,s_b+1,s_b+1, s_b+2,s_b+2,s_b+2,\dots, (s_b+1)(s_b+2)-1]
$$
represents exactly every integer $s\in S$ other than $s_b$, using Gauss's theorem that every integer is the sum of three triangular numbers.  

It therefore remains to show that $S_0$ is finite.  Since at each step there are only finitely many choices for $b_k$, the breadth at each node of the tree is finite, so it suffices to show that the supremum of the depth is finite.  To do so, we will consider all nodes at depth $4$.  Since the breadth is finite, there are only finitely many such nodes, and only finitely many leaves of depth less than or equal to $4$.  Therefore, it suffices to fix one such node and show that the depth of the resulting subtree is finite.  

Let $f$ be a triangular form of dimension at least $4$, and consider $Q_{\text{odd}}$ to be the corresponding quadratic form with $x_i$ odd.  Then the $n$-th coefficient of $f$ is the equal to the $8n+\ssum{i=1}{k}b_k$-th coefficient of $Q_{\text{odd}}$.  Therefore, the coefficients $a_f(n)$ are precisely the coefficients of the modular form corresponding to $Q_{\text{odd}}$.  Hence, we may decompose these coefficients into coefficients of an Eisenstein series plus coefficients of a cusp form.  Since the dimension of $f$ is at least $4$, the coefficients of the Eisenstein series grow faster than the coefficients of the cusp form as long as the Eisenstein series is non-zero.  We note that there are finitely many congruence classes where the coefficients of both are zero, namely those integers not locally represented by the quadratic form.  Hence, as long as the Eisenstein series is non-zero, there are only finitely many congruence classes and finitely many ``sporadic'' integers not represented by $f$, since the coefficients of the Eisenstein series are always positive.  It is simple to show that the Eisenstein series must be non-zero, however, since by Siegel's local density formula \cite{Siegel2}, this Eisenstein series is given by the difference of the local densities with arbitrary $x_i$ and the local densities with $x_i$ even for some $i$.  However, a quick check show that when $p\neq 2$ the local densities are the same, and, since the integer $8n+\ssum{i=1}{k}b_k$ is locally represented modulo 8 with $x_i$ all odd (namely $x_i=1$), the local density for $x_i$ arbitrary must be greater than when $x_i$ is even except at finitely many congruence classes.  

Let one of the nodes of our tree, $f\in T$ be given of exactly dimension $4$.  Then there are only finitely many congruence classes and finitely many ``sporadic'' integers not represented by $f$, and hence at each step of our escalation, our next choice of $s_b\in S$ must be one such integer.  We may only escalate finitely many times for the ``sporadic'' integers, and each time we escalate to include an element of one of the congruence classes, the congruence class is replaced with finitely many new ``sporadic'' integers at the next step.  Therefore, since there are only finitely many congruence classes, we can only add finitely many new ``sporadic'' integers overall, and hence the subtree of $f$ is of finite depth.
\end{proof}

\section{Determining $S_0$}\label{DetermineSection}
We will describe the algorithm to determine the set $S_0$.  For complete details of how to compute the bounds obtained given GRH, see \cite{Kane2} and \cite{Kane4}.    

\begin{proof}[Proof of Theorem \ref{DetermineS0Thm}]
From the proof in Section \ref{S0Section}, it is clear that $S_0$ is uniquely determined by the tree $T$, so this algorithm is equivalent to determining the tree $T$, since it is a simple check to determine the smallest $s\in S$ which is not represented by a fixed form $f$.  Constructing the tree as in the proof is also quite simple, so that the only remaining obstacle is determining whether a node of the tree is a leaf.  Our task is thus equivalent to determining (effectively) which integers are represented by $Q_{\text{odd}}$.  If the dimension (and hence the depth in the tree) of $f$ is at least $5$, then, using the trivial bounds for the coefficients of the Eisenstein series and the cusp forms, we may effectively determine a bound beyond which every integer locally represented is globally represented as in Tartakowsky's work \cite{Tartakowsky1}, and local conditions are a simple check at the primes dividing the discriminant.  For depth $4$, the wonderful optimal bound for cusp forms of Deligne \cite{Deligne1} and calculation of the anisotropic primes allows us to again effectively determine the set of integers represented by $Q_{\text{odd}}$ (cf. Hanke \cite{Hanke1}).  For forms of depth 2 or less, we note that $Q_{\text{odd}}$ only represents a set of density zero, so that the form cannot be a leaf.

It remains to determine whether a form of dimension $3$ is a leaf or not.  However, additional complications arise for ternary quadratic forms.  First note that the inclusion/exclusion of theta series 
$$
\theta_{Q_{\text{odd}}}:=\theta_{Q(x,y,z)}-\theta_{Q(2x,y,z)}-\theta_{Q(x,2y,z)}-\theta_{Q(x,y,2z)}+\theta_{Q(2x,2y,z)}+\theta_{Q(2x,y,2z)}+\theta_{Q(x,2y,2z)}-\theta_{Q(2x,2y,2z)}$$
has $n$-th coefficient non-zero if and only if $Q_{\text{odd}}$ represents $n$, where $Q$ is the quadratic form without the odd restriction.  We note that $\theta:=\theta_{Q_{\text{odd}}}$ decomposes as follows:
$$
\theta=(\theta-\theta_{\text{spin}}) + (\theta_{\text{spin}}-\theta_{\text{Gen}}) + (\theta_{\text{Gen}}),
$$
where $\theta_{\text{spin}}$ is the inclusion/exclusion of the (Siegel) weighted average over the spinor genus and $\theta_{text{Gen}}$ is the inclusion/exclusion of the weighted average over the genus.  It is well known (cf. \cite{DukePillot1}), that $\theta_Q-\theta_{\text{spin}(Q)}$ is the orthogonal complement of $U$, $\theta_{\text{spin}(Q)}-\theta_{\text{Gen}(Q)}\in U$, and $\theta_{\text{Gen}(Q)}$ is the Eisenstein series given by the local densities in \cite{Siegel2}.

Firstly, as noted in the introduction, the coefficients of the Eisenstein series grow like the class number, so that we have made the assumption of GRH for Dirichlet $L$-functions.  Additionally, we have assumed GRH for the $L$-series of weight 2 newforms.  Additionally, there are what are referred to as anisotropic primes.  That is, primes $p$ for which the $n$-th coefficient of the Eisenstein series does not grow when $n$ grows with high divisibility by $p$.  Luckily, a local condition allows one to check only finitely many primes (those dividing $2D$, where $D$ is the Discriminant) to determine these anisotropic primes.  Additionally, the genus of the quadratic form may be broken into what are referred to as Spinor Genera.  
In each spinor genus, there are finitely many integers $t$, called \begin{it}spinor exceptions\end{it} for which a subset of the square class $t\integer^2$ is not represented by the form.  This comes from the fact that the form splits naturally into three parts, namely an Eisenstein series, a cusp form in the space spanned by lifts of one dimensional theta series (we will denote this space by $U$), and a cusp form in the orthogonal complement of $U$.  The forms in the space spanned by the lifts of one dimensional theta series have coefficients which grow like $n^{\frac{1}{2}}$ in square classes $t\integer^2$, and hence with the same growth as the coefficients of the Eisenstein series.  However, the coefficients of each lift is non-zero exactly at $t$ times a square, and $t$ is restricted by the level of the corresponding modular form.  Schulze-Pillot has given an explicit algorithm to determine the full set of spinor exceptions for a given quadratic form \cite{SchulzePillot1}, so this problem will be resolved by using Schulze-Pillot's algorithm.  If the $t$-th coefficient of $\theta_{\text{spin}}$ is non-zero, then the growth of the coefficients of $\theta_{\text{spin}}$ grows like the class number in this square class.  Otherwise, by investigating the spinor norm mapping, Schulze-Pillot determines explicitly the integers $m$ such that the $tm^2$-th coefficient of $\theta_{\text{spin}}$ is zero and those for which the coefficient is equal to a positive constant times the Eisenstein series (cf. Schulze-Pillot \cite{SchulzePillot2}), for which we can reduce the problem to the case where $t$ is not a spinor exception for the genus.

Since we can determine the finitely many spinor exceptions, we may then ignore the part of the decomposition resulting from a cusp form in $U$, and it remains to give (effectively) a bound $N_{\theta}$ such that the $n$-th coefficient of the Eisenstein series are larger than the $n$-th coefficient of the cusp from in $U^{\perp}$ whenever $n>N_{\theta}$.

To do so, we must have effective bounds for the coefficients of both the Eisenstein series and the cusp form.  The Eisenstein series is
$$
C_{\theta} \ssum{d^2\mid n}{} h\left(-a_{(4D,n)}\frac{n}{d^2}\right)
$$
where $C_{\theta}$ may be determined by the local densities in \cite{Siegel2}, $h(m)$ is the class number of the imaginary quadratic field $\rational(\sqrt{-m})$ and $a_{(4D,n)}$ is a constant depending only on $\gcd(4D,n)$where $D$ is the discriminant \cite{Jones1}.

For square free integers, we will use Dirichlet's class number formula (cf. \cite{Davenport1} to rewrite the class number with the special value of the Dirichlet character $\chi_{-N}$,  
$$
h(-N)= \frac{\sqrt{n}L(\chi_{-N},1)}{\pi}.
$$
For square free integers, a celebrated result of Waldspurger \cite{Waldspurger1} allows us to rewrite the square of the absolute value of the coefficient of a Hecke eigenform as the special value of the (integral weight 2) Shimura lift \cite{Shimura1} of the Hecke eigenform.  Hence, we decompose the cusp form into Hecke eigenforms and then use the Schwarz's Inequality, giving 
$$
|a_{\theta_{\text{spin}}-\theta_{\text{Gen}}}(n)| \leq \sqrt{\ssum{i=1}{t}c_i \frac{L(G_i,\chi \chi_{-n},1)}{L(\chi_{-n},1)}}.
$$
where $c_i$ is some explicitly computable constant given by a fixed $n_0\equiv n\pmod{4D^2}$, $G_i$ is the Shimura lift, $\chi$ is the Nebentypus (cf. \cite{Ono1}) of the weight $3/2$ cusp form, and $L(G_i,\psi,1)$ is the special value of $G_i$ twisted by $\psi$.

Hence, rearranging, if $n$ is square-free such that the coefficient $a_{\theta}(n)$ is zero (and $n$ is not a spinor exception), then 
$$
c_{\theta} n^{\frac{1}{2}}\leq  \ssum{i=1}{t} c_i \frac{L(G_i,\chi \chi_{-n},1)}{L(\chi_{-n}\chi_{a_{(D,n)}},1)^2},
$$
where $c_{\theta}$ is an explicitly computable constant.  Thus, it only remains to bound $\frac{L(G_i,\psi_1,1)}{L(\psi_2,1)^2}\ll_{\delta} n^{\delta}$ effectively and with the implied constant given explicitly for some $\delta<\frac{1}{2}$.  However, under the given GRH assumptions, an explicit bound is given in \cite{Kane4}, and the details of the ensuing calculation for the implied constant are given in \cite{Kane2}.  

For $n$ not square free, the Hecke operators may be used to show an explicit bound for the squares part beyond which integers must be represented, away from the spinor exceptions.  For more details, please see \cite{Kane1} or \cite{Kane4}.  Therefore, we can conclude that the set of square free integers not represented by this form may be exactly determined by checking up to the bound obtained, and hence we may (effectively and efficiently) determine whether this form is a leaf.  For forms with small discriminant (say, less than 300), the bound obtained is often well below $10^{12}$, and hence is well within the computing power of current technology.

\end{proof}
\begin{remark}
In practice, whenever a leaf exists at depth 3, we will determine in general the set of integers not represented by $Q_{\text{odd}}$ for the nodes at depth $3$ (not just the leaves) and then note which of these integers remain at each step of the escalation, instead of using the arguments of Tartakowsky \cite{Tartakowsky1} at depth at least 5 or the bounds of Deligne \cite{Deligne1} at depth 4.  
\end{remark}

We now implement the above algorithm to show the set $S_0$ given in Conjecture \ref{OddConjecture}, when $S$ is the set of all odd integers, is the correct smallest such set under GRH. 
\begin{proof}[Proof of Theorem \ref{OddThm}]
We will proceed by considering each node at depth $3$, and determining the corresponding subtree under these GRH assumptions.  For notational ease, we will refer to the form 
$$
f:=\ssum{i=1}{r}b_i T_{x_i}
$$
by $[b_1,\dots, b_r]$, and the corresponding quadratic form by $(b_1,\dots,b_k)$.

The forms $[1,1,1]$, $[1,1,2]$, $[1,1,4]$, $[1,1,5]$, $[1,2,2]$, $[1,2,3]$, and $[1,2,4]$ represent every natural number by Liouville's theorem.  Since $[1,1,3]$ is a genus one form, the integers not represented by $[1,1,3]$ are precisely the integers $n$ such that 
$$
8n+5=3^{2r+1}(3\ell+2).
$$
Therefore, we are missing the integer $17$, and we must escalate to $[1,1,3,k]$ for some $k\leq 17$.  If $n$ is not represented by $[1,1,3,k]$, then it is not represented by $[1,1,3]$, either, so $8n+5=3^{2r+1}(3\ell+2)$, so that $n\equiv 5\pmod{9}$ or $n\equiv 8\pmod{9}$, depending on whether $r>0$ or $r=0$, respectively.  But then, taking $x_4=1$, it follows for $n>k$ that $n-k\equiv 5\pmod{9}$ or $n-k\equiv 8\pmod{9}$, and hence it follows that $3\mid k$.  If $r>0$ then it follows that $k\equiv 0,6\pmod{9}$, and if $r=0$ it follows that $k\equiv 0,3\pmod{9}$.  

For the case $[1,1,3,3]$ we then note that the form $Q=(1,1,3,3)$ is genus 1.  Our inclusion/exclusion of theta series gives that $n$ is represented if and only if the $2n+2$-th coefficient of $\theta_Q$ is positive, and it follows that every $n$ is represented because the local conditions are always satisfied.  

For the cases $[1,1,3,k]$ with $k=6$ or $k=15$ we have $r>0$, so that we only need to consider $n\equiv 5\pmod{9}$, or in other words $8n+5=3^{2r+1}(3\ell+2)$ with $r>0$.  We check the cases $n\leq 3k$ by hand and for $n>3k$, the choice $x_4=2$ shows that $8(n-3k)+5=3^{2r'+1}(3\ell'+2)$, with $r'>0$ by congruence conditions modulo $9$.  Taking the difference and denoting $R=\min\{ r,r'\}$, we have 
\begin{multline*}
24k = 8n+5 - (8(n-3k)+5) = 3^{2r+1}(3\ell+2)-3^{2r'+1}(3\ell'+2) \\
= 3^{2R+1}\left( 3^{2(r-R)+1}(3\ell+2)-3^{2(r'-R)+1}(3\ell'+2)\right).
\end{multline*}
But $v_3(24k)=2$ and $3^{2R+1}$ divides the right hand side, giving a contradiction since $R>0$.

In the case $[1,1,3,12]$ we note that we have the truant $89$ and $r=0$ from above.  In this case we take $x_4=1$ to obtain $n'=8n+5=3^{2r'+1}(3\ell'+2)+ 96$ with $r'>0$.  We then have either $r'=1$ and $n'\equiv 69\pmod{81}$ or $r'>1$ and $n'\equiv 15\pmod{81}$.  Assume $r'>1$ and set $x_4=4$ to obtain $n'=3^{2r''+1}(3\ell''+2)+960$.  Taking the difference, we get 
$$
54\equiv 864 = 3^{2r'+1}(3\ell'+2) - 3^{2r''+1}(3\ell''+2) \equiv - 3^{2r''+1}(3\ell''+2)\pmod{81}.
$$
It follows immediately that $r''=0$ because otherwise the right hand side would be zero.  But now we have $n'=3^3(3\ell''+2)+960\equiv 42\pmod{81}$, which contradicts the fact that $n'\equiv 15\pmod{81}$.  Hence we have only the case $n'\equiv 69\pmod{81}$ remaining.  We now escalate to $[1,1,3,12,k]$ for $9\leq k\leq 89$ (although we have restricted our coefficients to be monotone, we include the case $k=9$ for usage below).  Since $[1,1,3,12]$ represents every integer not congruent to 69 modulo 81, we are done when $k\neq 81$, and for $k=81$ the truant 89 remains.  Furthermore, if we escalate further with $k=81$ we will never obtain $89$ since $[k,k,k]$ represents precisely $k\nature$ by Gauss's theorem.  Furthermore, if we ever choose $k=81$ and then escalate to another integer (such as $[1,1,3,12,81,k]$) we do not obtain a new truant because $[1,1,3,12,k]$ has no truants.  Whenever this situation occurs henceforth we shall say that we are ``stuck'' at $k=81$.  Following the above, for $[1,1,3,9]$ we are stuck at $k=9$.  This concludes the subtree of $[1,1,3]$, and since $[1,1]$ does not represent $5$, we have included the entire subtree of $[1,1]$.

Our arguments for $[1,2,5]$, $[1,2,10]$, $[1,3,4]$, and $[1,4,6]$ will all be identical and similar to the cases above, so we will combine them together.  We will demonstrate the argument for $[1,2,5]$ and leave the other cases to the reader.  First we note that the number of representations of $n$ by $[1,2,5]$ is the same as the number of representations of $8n+8$ by $(1,2,5)$ minus the number of representations of $8n+8$ by $(4,8,20)$, since if any are even then all must be even, taking everything modulo $8$.  For simplicity, we will denote $t_{[b]}(n)$ to be the number of times that $n$ is represented by the sum of triangular numbers corresponding to $b$, $r_{(b)}(n)$ for the number of times the quadratic form represents $n$, and $r_{(b)}^{o}(n)$ for the number of times the quadratic form with all $x_i$ odd represents $n$.  So the above is simply
$$
t_{[1,2,5]}(n)=r_{(1,2,5)}^{o}(8n+8)=r_{(1,2,5)}(8n+8)-r_{(4,8,20)}(8n+8).
$$
Now, $r_{(4,8,20)}(8n+8)=r_{(1,2,5)}(2n+2)$.  However, $(1,2,5)$ is genus $1$, so $r_{(1,2,5)}(m)$ is given precisely by the local density \cite{Siegel2}.  Checking the local densities, we see that $r_{(1,2,5)}(8n+8)=2r_{(1,2,5)}(2n+2)$, since the local densities are clearly equal when $p\neq 2$ (taking the isomorphisms $x_i\mapsto 2^{-1}x_i$), and for $p=2$ a simple computation shows exactly twice as many solutions modulo the same 2 power.  Therefore, 
$$
t_{[1,2,5]}(n)=r_{(1,2,5)}(2n+2).
$$
Again, noting that $(1,2,5)$ is genus 1, we know that $n$ is represented globally if and only if it is represented locally.  The integers not represented locally by $(1,2,5)$ are integers of the form $5^{2r+1}(5n+m)$ where $m$ is a non-square modulo $5$.  Therefore, it follows that if $n$ is not represented by $[1,2,5]$, then $5\mid (n+1)$.  So $[1,2,5,k]$ must be a leaf if $5\nmid k$.  Since the truant for $[1,2,5]$ is $19$, we need only check $k=5$, $k=10$, and $k=15$.  We are stuck at $k=15$, so we only need to show the cases $k=5$ and $k=10$.

Let $m$ be smallest such that $[1,2,5,5]$ does not represent $m$.  Then $5\mid m+1$ from above and $\frac{2m+2}{5^{2^r+1}}$ is not a square modulo $5$, taking $x_4=0$.  If $r\neq 0$, then take $x_4=2$, and note that $5^{2^r+1}(5n+\not\square) -2\cdot 3\cdot 5 \equiv  20\pmod 25$, so that $m-15$ is represented by $[1,2,5]$, and hence $m$ is represented by $[1,2,5,5]$ as long as $m\geq 15$, which is as desired.  Thus, we only need to consider $r=0$.  The non-squares modulo $5$ are $2$ and $3$.  As above, taking $x_4=2$ when $\frac{2m+2}{5}\equiv 2\pmod 5$ and $x_4=1$ when $\frac{2m+2}{5}\equiv 3 \pmod 5$ gives the desired result.  

For $k=10$, we again conclude that $r\neq 0$ by taking $x_4=1$.  Furthermore, if $\frac{2m+2}{5}\equiv 3\pmod 5$, then $x_4=1$ again gives us the desired conclusion.  Hence, only the case $\frac{2m+2}{5}\equiv 3\pmod 5$ remains.  For $[1,2,5,10,k]$, with $5\mid k$, we may again conclude that by taking $x_5=0$ and $x_5=1$ that $25\mid k$ since $2(m+2)= 5(5n+2)$ and $2(m+2-k)= 5(5n'+2)$, so $2k=25(n-n')$.  Since $29$ is the truant of $[1,2,5,10]$, the result follows when $k\neq 25$. But we are stuck at $k=25$, so we have the desired result.

We will leave out the analogous proofs for $[1,2,10]$, $[1,3,4]$, and $[1,4,6]$, but list the truants from their subtrees for completeness.  The truants coming from $[1,2,10]$ are $29$ and $49$ (from $[1,2,10,20]$).  The only truant from the subtree of $[1,3,4]$ is $11$.  The truants from the $[1,4,6]$ subtree are $17$, $29$ (from $[1,4,6,12]$, and $35$ (from $[1,4,6,6]$).  

We will now show the subtree for $[1,4,4]$.  While $(1,4,4)$ is not genus 1, Benham, Earnest, Hsia, and Hung have shown that it is spinor genus 1 \cite{BenhamEarnestHsia1}.  Moreover they have shown that $(1,4,16)$ is spinor genus 1 by showing that the other member of its genus, namely $4x^2+4y^2+5z^2+4xz$ is spinor genus $1$.  Therefore, the difference $r_{(1,4,4)}(8n+9)-r_{(1,4,16)}(8n+9)$ can be decomposed into coefficients of the Siegel averaging of the genus, and a cusp form in $U$ which has nonzero coefficients only at finitely many square classes.  Using Schulze-Pillot's classification \cite{SchulzePillot1} or the generalization of Earnest, Hsia, and Hung \cite{EarnestHsia1} to determine all $t$ such that the square classe $t\integer^2$ has nonzero coefficients for the resulting cusp forms in $U$ with these two quadratic forms, we conclude that only $t=1$ occurs.  Therefore, it follows that if $m$ is represented by $[1,4,4]$, then $8m+9$ must be a square.  The first truant is $m=35$, so we consider $[1,4,4,k]$ for $k\leq 35$.  Let $m$ not represnted by $[1,4,4,k]$ be given.  Then $m-k$ is also not represented by $[1,4,4]$, so that $8m+9$ is a square, say $s^2$ and $8m+9-8k$ is a square, say $t^2$.  Therefore, $s^2-t^2=8k$.  But the difference between $s^2$ and $t^2$ must be at least the difference between $s^2$ and $(s-1)^2$, which is $2s-1$.  Therefore, $2s-1\leq 8k$.  This restricts the possible choices for $s$ to a (small) finite set, and hence the possible choices for $m$.  Checking each such choice of $s$ for each $k$ allows us to determine the integers not represented by $[1,4,4,k]$.  We conclude that we are done for every integer other than $k=15$ and $k=33$.  For $k=15$, we represent exactly every integer other than $2$ and $35$, so $[1,4,4,15,k']$ will represent every integer except when $k'=33$.  For $k=33$ or $k'=33$, we are stuck at $35$, and hence we are done with the $[1,4,4]$ subtree.

We will begin to use our GRH assumptions now.  We will indicate clearly where we have made the assumptions, and furthermore we will indicate the cases where these assumptions were seemingly unavoidable.

The eight cases $[1,2,6]$, $[1,2,8]$, $[1,2,9]$, $[1,2,11]$, $[1,4,5]$, $[1,4,8]$, $[1,4,9]$, and $[1,5,6]$ will all follow analogous arguments.  The cases $[1,2,6]$, $[1,2,9]$, and $[1,4,5]$ are the three cases where the GRH assumptions were seemingly unavoidable.  In each case, we will be able to decompose the theta series for the corresponding quadratic forms with odd conditions into the Eisenstein series plus a Hecke eigenform.  The Hecke eigenform in each case is in the complement of the space spanned by lifts of one dimensional theta series.  Due to the fact that these are all genus 2 quadratic forms, we are able to first obtain the following proposition.

\begin{proposition}\label{equivalentprop}
\begin{table}
\caption{Equivalent Quadratic Forms}
\centering
\begin{tabsize}
\begin{tabular}{clccclc}
\toprule
Triangular Form & represents&$n$ & $\iff$ & Quadratic Form & represents & $m$\\
\midrule
$[1,2,6]$ && $n$ & & $(2,4,7,0,0,4)$ & &$8n+9$\\
$[1,2,8]$ & &$n$ & & $(2,4,9,0,0,4)$ & &$8n+11$\\
$[1,2,9]$ & &$n$ & & $(2,3,4,2,0,2)$ & &$2n+3$\\
$[1,2,11]$ & &$n$ & & $(1,6,8,0,0,4)$ && $4n+7$\\
$[1,4,5]$ & &$n$ & & $(1,4,5,0,0,0)$ & &$8n+10$\\
$[1,4,8]$ & &$n$ & & $(4,4,9,0,4,0)$ & &$8n+13$\\
$[1,4,9]$ & &$n$ & & $(1,4,9,0,0,0)$ & &$8n+14$\\
$[1,5,6]$ & & $n$ & & $(3,3,4,0,2,2)$ && $2n+3$\\
\bottomrule
\end{tabular}
\end{tabsize}
\label{EquivalentTable}
\end{table}

Each quadratic form $[b_1,b_2,b_3]$ in Table \ref{EquivalentTable} represents $n$ if and only if the corresponding quadratic form 
$$
ax^2+by^2+cz^2+dxy+exz+fyz
$$
represents $m$.  Here we do not have the ``odd'' condition for the quadratic forms.
\end{proposition}
\begin{proof}
Each of the above assertions follows the same simple argument, which we will demonstrate explicitly for the form $[1,2,11]$.  

First, note that if not all of $x,y,$ and $z$ are odd, then $x^2+2y^2+11z^2=8n+14$ has a solution modulo $8$ only if $x$ and $z$ are both even.  Therefore, if $t_b(n)$ is the number of solutions of the triangular form represented by $b$, and $r_Q(m)$ is the number of solutions to $Q(x)=m$, then
\begin{multline*}
t_{[1,2,11]}(n)=r_{(1,2,11,0,0,0)}(8n+14)-r_{(4,2,44,0,0,0)}(8n+14)\\
 = r_{(1,2,11,0,0,0)}(8n+14)-r_{(1,2,22,0,0,0)}(4n+7).
\end{multline*} 
Now we note that $(1,2,11,0,0,0)$ is a genus 2 quadratic form.  The other representative of the genus is $(2,3,4,0,0 2)$.  If $2x^2+3y^2+4z^2+2yz=8n+14$ has a solution, then it is clear that $y$ must be even.  Therefore, 
$$
r_{(2,3,4,0,0,2)}(8n+14)=r_{(2,12,4,0,0,4)}(8n+14)=r_{(1,2,6,0,0,2)}(4n+7).
$$
But $(1,2,6,0,0,2)$ is a genus 1 quadratic form, so it follows that $r_{(1,2,6,0,0,2)}(4n+7)$ is merely the value given by Siegel's local density formula as in \cite{Siegel2}.  Using this observation and the fact that the local densities are equal for $(1,2,6,0,0,2)$ and $(2,3,4,0,0,2)$, it follows that $r_{(1,2,6,0,0,2)}(4n+7)=r_{(2,3,4,0,0,2)}(8n+14)=r_{(1,2,11,0,0,0)}(8n+14)$.

The form $(1,2,22,0,0,0)$ is again genus $2$, and the other representative of the genus is $(1,6,8,0,0,4)$.  Therefore, the theta series for $Q=(1,2,22,0,0,0)$,
$$
\theta_{Q}(z):=\ssum{x}{}q^{Q(x)}
$$
satisfies $\theta_{Q}=E_Q+g$, where $E_Q$ is the Eisenstein series obtained by taking the local densities and $g$ is a cusp form which is a Hecke eigenform.  Siegel formula shows for $Q':=(1,6,8,0,0,4)$ that  
$$
E_Q=\frac{ \frac{\theta_{Q}}{4} + \frac{\theta_{Q'}}{2}}{\frac{3}{4}}.
$$
Therefore, $\theta_{Q'}=E_Q - g$.  By obversing that the local densities for $(1,2,11,0,0,0)$ and $(4,2,44,0,0,0)$ are the same other than at $p=2$, one can easily see by explicitly computing the local density at $p=2$ that $a_{E_{(1,2,11,0,0,0)}}(8n+14)=2 a_{E_{1,2,22,0,0,0}}(4n+7)$.  Therefore, we have shown that
\begin{multline*}
r_{(1,2,11,0,0,0)}(8n+14)-r_{(1,2,22,0,0,0)}(4n+7)= 2 a_{E_{1,2,22,0,0,0}}(4n+7) - (a_{E_{1,2,22,0,0,0}}(4n+7)+g)\\
 = a_{E_{1,2,22,0,0,0}}(4n+7)-g=r_{(1,6,8,0,0,4)}(4n+7).
\end{multline*}
This is precisely what we wanted to show.  The other cases follow analogously.
\end{proof}

We now proceed by determining which integers in these arithmetic progressions are represented by these quadratic forms.  Since each of these forms is genus 2, as well as spinor genus 2, we know that $\theta=E+g$, where $g$ is a Hecke eigenform in the complement of the space spanned by lifts of one dimensional theta-series.  Thus, we will employ the argument given in \cite{Kane4} in the case where there is precisely one Hecke eigenform.

We will begin by constructing an algorithm that given a non-zero Hecke eigenform $g$, a constant $c_E$ depending only on the local densities, such that the Eisenstein series has coefficients $a_E(N)=c_E h(mN)$ for some fixed integer $m$ and $N$ square free, where $h(D)$ is the class number, the modulus $q$ such that the corresponding twist from Waldspurger's theorem \cite{Waldspurger1} of the Shimura lift $G$ of $g$ is of modulus $qN^2$, the integer $m$ given above in the class number, an integer $D_0$ such that $a_g(D_0)$ is nonzero, $a_g(D_0)$, and the character $\chi$ that we are twisting by, and returning a bound $D$ beyond which the Eisenstein series is nonzero in the congruence class corresponding to $D_0$.

We do so by fixing $\x:=455$, $\sigma:=1.1573$, and $\sigma_2:=1.3465$ and calculating the bounds given in \cite{Kane4} for $L(\chi_{mN},\sigma+it)$ and $L(G,\chi,\sigma+it)$, where $G$ is the Shimura lift of $g$.  Bounding $\alpha$, $\beta$, $\gamma$, and $\delta$ (these are all independent of the particular choice of $g$) as in Lemma 7.1 of \cite{Kane4}, we see that 
$$
\begin{array}{ll}
\alpha(\x)\leq 0.089028567932572,& \beta(\x)\geq 0.0886630818642167,\\
\gamma(\x)\leq 0.249235264918139,&\delta(\x)\leq 0.0963544917482776.
\end{array}
$$
Given these bounds, the rest of the constants are easy to calculate (We use here the computer algebra system MAGMA \cite{BosmaCannon1}.), and further details may be found in \cite{Kane2}.  

Notice that the bound obtained by this algorithm is only valid in the congruence class congruent to $D_0$ modulo $8$ times the square of the determinant of the corresponding quadratic form, and that for each congruence class the choice of $m$ and $\chi$ may vary.  We therefore will run the algorithm for a choice of $D_0$ in each congruence class which satisfies the congruence modulo a 2 power given above, and we will merely state the largest such bound obtained in Table \ref{SufficientTable}.

\begin{table}
\caption{Sufficient Bounds under GRH}
\centering
\begin{tabsize}
\begin{tabular}{lcc}
\toprule
Triangular Form & Bound $D_0$ for Quadratic Form    & Bound $D_0'$ for odd squares\\
\midrule
$[1,2,6]$       & $1.23\times 10^9$                 &   $1.23\times 10^9$\\
$[1,2,8]$       & $6.0\times 10^8$                  &   $6.0\times 10^8$\\
$[1,2,9]$       & $1.68\times 10^6$                 &   $ 6.72\times 10^6$\\
$[1,2,11]$      & $8.0\times 10^4$                  &   $1.6\times 10^5$ \\
$[1,4,5]$       & $2.6\times 10^5$                  &   $2.6\times 10^5$ \\
$[1,4,8]$       & $5.7\times 10^8$                  &   $5.7\times 10^8$ \\
$[1,4,9]$       & $1.1\times 10^{12}$               &   $1.1\times 10^{12}$ \\
$[1,5,6]$       & $4.55\times 10^{9}$               &   $1.82\times 10^{10}$\\
\bottomrule
\end{tabular}
\end{tabsize}
\label{SufficientTable}
\end{table}

We then check up to the bound $D_0'$ for odd squares with a computer (using the fact that it is diagonal to our advantage by splitting off one dimension, as in \cite{Kane2}) and list the integers not represented by each triangular form.  We will list the triangular form, then the congruence classes not represented locally by the form, and finally the finite list of ``sporadic'' integers not represented globally but represented locally by the form.  Note that we have a leaf (and hence GRH seems unavoidable with current techniques) if and only if there are no congruence classes and no sporadic odd integers, namely the cases $[1,2,6]$, $[1,2,9]$, and $[1,4,5]$.

\begin{table}
\caption{Exceptional Integers Not Represented}
\centering
\begin{tabsize}
\begin{tabular}{lcc}
\toprule
Triangular Form & Congruence Classes    & Sporadic Integers\\
\midrule
$[1,2,6]$       &  $\emptyset$          & $\emptyset$    \\
$[1,2,8]$       &  $\emptyset$          & $\{4,19,112\}$      \\
$[1,2,9]$       &  $\emptyset$          & $\emptyset$ \\
$[1,2,11]$      &  $\emptyset$          & $\{4,25\}$\\
$[1,4,5]$       &  $\emptyset$          & $\{2,26,38\}$ \\
$[1,4,8]$       &  $\emptyset$          & $\{2,16,17\}$ \\
$[1,4,9]$       & $m\equiv 2\mod 9$ and $m\equiv 8\mod 9$  &  $\emptyset$ \\
$[1,5,6]$       & $\emptyset$    & $\{2,4,13,35\}$\\
\bottomrule
\end{tabular}
\end{tabsize}
\label{ExceptionalTable}
\end{table}

Given table \ref{ExceptionalTable}, it is easy to see that the only truants which arise from all of the subtrees other than $[1,4,9]$ are $13$, $17$, $19$, and $25$.  The first odd integer not represented by $[1,4,9]$ is $11$, so we only need to consider $[1,4,9,k]$ for $9\leq k\leq 11$.  The choice $k=9$ is stuck at $11$, and $k=10$ or $k=11$ clearly represent every integer other than $2$ and $8$, by taking $x_4=0$ or $x_4=1$.  

The case $[1,2,7,k]$ with $8\leq k\leq 11$ also follows from the above cases with no , and $k=7$ is stuck at $11$ so we get no new truants.  The case $[1,4,7]$ analogously gives no new truants, and $[1,3,3]$ and $[1,5,5]$ will follow similarly after we show the argument for $[1,3,5]$ and $[1,5,7]$.  Thus, it only remains to show the subtrees for $[1,3,5]$ and $[1,5,7]$.  

For $[1,3,5]$, we will follow a similar argument to above, but we must be slightly more careful.  We note that
$$
t_{[1,3,5]}(n) = r_{(1,3,5,0,0,0)}(8n+9) -  r_{(1,5,12,0,0,0)}(8n+9).
$$
The theta-series for $(1,3,5,0,0,0)$ decomposes as $E+g_1$, while the theta-series for $(1,5,12,0,0,0)$ decomposes as $\frac{1}{2}E+g_2$.  In this case, the Shimuar lift of $g_1-g_2$ is $\frac{1}{2}(-G_{15}+G_{15}|V(2)-4G_{15}|V(4))$, where $G_{15}$ is the (unique) newform of weight 2 and level 15 and $V(d)$ is the $d$-th V-operator (cf. \cite{Ono1}).  Therefore, since the Hecke operators commute with the Shimura lift, it follows that $g_1-g_2$ is a Hecke eigenform, so we may use the above algorithm with $g=g_1-g_2$.  

When $5$ does not divide $8n+9$, our algorithm gives the bound $D_0'=9.4\times 10^8$.  Checking up to this bound we find no sporadic integers outside of $n\equiv 2\mod 5$.  Therefore, all integers other than a subset of $n\equiv 2 \mod 5$ are represented by $[1,3,5]$.  However, since the truant of $[1,3,5]$ is $7$, we only need to consider $[1,3,5,k]$ for $5\leq k\leq 7$, and $5$ is stuck on $7$.  Hence, using $x_4=0$ or $x_4=1$ we represent every integer other than $2$.  The reason for separating the case of $n\equiv 2 \mod 5$ is because the bound obtained by the algorithm was not computationally feasible and was not needed to obtain the desired result.  

In the case of $[1,5,7]$, 
$$
t_{[1,5,7]}(n) = r_{(1,5,7,0,0,0)}(8n+13)-r_{(1,5,28,0,0,0)}(8n+13),
$$
and both $(1,5,7,0,0,0)$ and $(1,5,28,0,0,0)$ are genus $3$ and spinor genus $3$.  Therefore, $(1,5,7,0,0,0)$ decomposes into
$$
E+g_1+g_2
$$
and $(1,5,28,0,0,0)$ decomposes into
$$
\frac{1}{2}E+g_3+g_4,
$$
where $g_1$, $g_2$, $g_3$, and $g_4$ are Hecke eigenforms in the complement of the space spanned by lifts of one dimensional theta-series.  Moreover, computation shows that the Shimura lift of $g_1-g_3$ is $\frac{1+\sqrt{17}}{2\sqrt{17}} (G_{35} + c_1 G_{35}|V(2)+c_2 G_{35}|V(4))$, where $G_{35}$ is the newform of weight 2 and level $35$ whose second coefficient is $\frac{-1+\sqrt{17}}{2}$.  The constants $c_1$ and $c_2$ are irrelevant, since we will twist the Hecke eigenform by $\chi_4$, killing all of the coefficients divisible by $2$, and the coefficients of $G|V(d)$ only has coefficients divisible by $d$.  Moreover, the Shimura lift of $g_2-g_4$ is $(1-\frac{1+\sqrt{17}}{2\sqrt{17}})(\sigma{G_{35}} + \sigma(c_1) \sigma{G_{35}}|V(2)+\sigma(c_2) \sigma{G_{35}}|V(4))$.  Here $\sigma$ is the Galois map sending $\sqrt{17}$ to $-\sqrt{17}$.

We could write a separate algorithm from the one above for two (or any arbitrary number of) eigenforms, but for simplicity we will simply rewrite the above sum as
$$
(\alpha E+(g_1-g_3))+ \left(\left(\frac{1}{2}-\alpha\right)E +(g_2-g_4)\right)
$$
for some appropriate choice of $\alpha$, and then bound each half separately as above, taking the maximum of the two bounds, since beyond the corresponding maximum bound $D_0$, $a_{\alpha E+(g_1-g_3)}>0$ and $a_{\left(\frac{1}{2}-\alpha\right)E +(g_2-g_4)}>0$.  The optimal choice of $\alpha$ would give equal $D_0$ bounds for each.  

We choose $\alpha=\frac{17}{36}$ here, so that $\frac{1}{2}-\alpha=\frac{1}{36}$.  Doing so, whenever $5$ does not divide $8n+13$, we get a bound for the part corresponding to $g_1-g_3$ of $D_0'=4.53\times 10^8$ and a bound of $\tilde{D_0}'=1.65\times 10^8$ for the part corresponding to $g_2-g_4$.  The only sporadic integer not congruent to $4$ mod $5$ is $2$.  Therefore, we are again done, since $[1,5,7,k]$ represents every odd when $8\leq k\leq 9$, as $n\equiv 4\mod 5$ must be represented with $x_4=1$, and $k=7$ is stuck at $9$.

\end{proof}

\begin{remark}
One may naturally ask here whether other currently known techniques are sufficient to determine the set $S_0$ for $S$ the set of all odd integers.  This question, it turns out, is more easily asked than answered.  We will give a partial answer to this question, however.  For a node of depth $3$, determining whether the node is a leaf when the node is not of genus 1 or spinor genus 1 does not seem possible given the current techniques.  Indeed, if such an algorithm were known, it would most likely also give a solution for integers represented by Ramanujan's ternary quadratic form, and hence give an alternate proof of Ono and Soundararajan's result without GRH.  

There appear to be 3 leafs of depth $3$ which are not (spinor) genus 1 when $S$ is the set of all odd integers, as no counterexamples have been found up to the bound which is predicted by GRH.  Therefore, if current techniques were sufficient to determine $S_0$, then these ``leaves'' must indeed not truly be leaves.  But then there must exist (large) counterexamples, and our set $S_0$ must be incorrect.  For this reason, it is not this author's belief that current techniques will solve this problem without GRH unless a new breakthrough occurs in the area.  Moreover, since all of the other counterexamples for the other nodes (of greater depth) provably occur at very small integers, it seems reasonable to conjecture that $S_0$ is as given.
\end{remark}

We conclude with the proof of Theorem \ref{curiousthm}.
\begin{proof}
The set $S$ is precisely the integers represented by $[2,3,4]$.  The integers less than or equal to 89 not represented by $[2,3,4]$ are $1,8,31$.  Since the first truant of $S$ is $2$, we begin with $[1]$ or $[2]$.  

The truant of $[1]$ is $2$, so we obtain $[1,1]$ and $[1,2]$.  For $[1,2]$ the next truant is $4$, and the remaining subtree has no further truants by Liouville's theorem.  For $[1,1]$ the truant is $5$.  Each case other than $[1,1,3]$ is then covered by Liouville's theorem, while the case $[1,1,3]$ has the same subtree as in Theorem \ref{OddThm}, since in each case the truants were the smallest integer other than 8 not represented by the form and $31$ is not one of the truants.  From this we obtain the truants $17$ and $89$.  This concludes the $[1]$ subtree.

For the $[2]$ subtree the first truant is $3$ so we escalate to $[2,2]$ or $[2,3]$.  The case $[2,2]$ still has truant $3$, so we escalate to $[2,2,2,3]$ and $[2,2,3]$.  Clearly $[2,2,2,3]$ represents every natural number greater than $1$ because $[2,2,2]$ represents every even natural number.  The form $(2,2,3)$ is genus 1 so $[2,2,3]$ represents every natural number $n$ such that $8n+7$ is not of the form $3^{2r+1}(3m+2)$.  The first truant is $10$ and if $3\nmid k$, we are clearly done for $[2,2,3,k]$, so we only have $k=3,6,9$ and we are stuck at $k=9$.  For $k=3$ we have the truant $19$ and for $k=6$ we have the truant $19$.  Analogous to the case $[1,1,3]$ we are able to show that no other truants occur in this subtree.  The truant of $[2,3]$ is $4$, so we have $[2,3,3]$ or $[2,3,4]$.  We are stuck at $k=3$ and $[2,3,4]$ represents $S$ by definition.
\end{proof}

\section*{Acknowledgements}
The author would like to thank W. Bosma for guidance and helpful conversation.


\begin{thebibliography}{99}
\bibitem{BenhamEarnestHsia1} J.W. Benham,  A. G. Earnest, J. S. Hsia, D.C. Hung, Spinor regular positive ternary quadratic forms, J. London Math. Soc. (2), no. 1, 1--10, 1990. 
\bibitem{Bhargava1} M. Bhargava, On the Conway-Schneeberger fifteen theorem, Quadratic Forms and their applications (Dublin, 1999), 27-37, Contemp. Math 272, Amer. Math. Soc., Providence, RI, 2000.
\bibitem{BhargavaHanke1} M. Bhargava, J. Hanke, Universal Quadratic Forms and the 290-Theorem, Invent. Math., to appear.
\bibitem{BosmaCannon1} Bosma, W., Cannon, J., Playoust, C., The Magma algebra system I.  The user language.  Computational algebra and number theory, J. Symbolic Comput., 24, 235-265, 1997. 
\bibitem{Davenport1} H. Davenport, Multiplicative Number Theory, Springer Verlag, New York, 1980.
\bibitem{Deligne1} P. Deligne, La Conjecture de Weil I, Inst. Hautes \'Etudes Sci. Publ. Math., 43, 273-307, 1974.
\bibitem{Duke1} B. Duke, Hyperbolic distribution problems and half-integral weight Maass forms, Invent. Math., 92, 73-90 (1988).
\bibitem{DukePillot1} B. Duke, R. Schulze-Pillot, Representations of Integers by Positive Ternary Quadratic Forms and Equidistribution of Lattice Points on Ellipsoids, Invent. Math., 99, 49-57, 1990.
\bibitem{EarnestHsia1} A.G. Earnest, J.S.  Hsia, D.C.  Hung, Primitive representations by spinor genera of ternary quadratic forms.  J. London Math. Soc. (2)  50,  no. 2, 222--230, 1994.
\bibitem{Hanke1} J. Hanke, Some Recent Results about (Ternary) Quadratic Forms, Number Theory, CRM Proc. and Lect. Notes, Vol. 36, 147-164, 2004. 
\bibitem{Jones1} B. Jones, The Regularity of a Genus of Postiive Ternary Quadratic Forms, Trans. Amer. Math. Soc., 111-124, 1931.
\bibitem{Kane1} B. Kane, PhD. Thesis, University of Wisconsin, May 2007.
\bibitem{Kane2} B. Kane, CM Liftings of Supersingular Elliptic Curves, preprint.
\bibitem{Kane3} B. Kane, The Triangular Theorem of Eight, preprint.
\bibitem{Kane4} B. Kane, Representations of Integers by Ternary Quadratic Forms, International Journal of Number Theory, to appear. 
\bibitem{Oesterle1} Oesterl\'e, J. "Nombres de classes des corps quadratiques imaginaires." Ast\'erique 121-122, 309-323, 1985.
\bibitem{Ono1} K. Ono, Web of Modularity:  Arithmetic of the Coefficients of Modular Forms and $Q$-series, Amer. Math. Society, 2003.
\bibitem{OnoSound1} K. Ono, K. Soundararajan, Ramanujan's ternary quadratic form, Invent. Math., 130, 415-454 (1997).
\bibitem{SchulzePillot1} R. Schulze-Pillot, Rainer, Darstellungsmasse von Spinorgeschlechtern tern\"arer quadratischer Formen., 
J. Reine Angew. Math. 352, 114-132, 1984.
\bibitem{SchulzePillot2} R. Schulze-Pillot, Exceptional Integers for Genera of Integral Ternary Positive Definite Quadratic Forms, Duke Math. Journal, 102, No. 2, 351-357, 2000.
\bibitem{Shimura1} G. Shimura, On Modular Forms of Half Integer Weight, Ann. of Math., 97, 440-481, 1973.
\bibitem{Siegel1} C. Siegel, \"Uber Die Klassenzahl Quadratischer Zahlk\"orper, Acta Arith., 1, 83-86, 1936.
\bibitem{Siegel2} C. Siegel, Über die analytische Theorie der quadratischen Formen. Ann. Math. 36, pp. 527606, 1935.
\bibitem{Tartakowsky1} W. Tartakowsky, Die Gesamtheit der Zahlen, die durch eine quadratische form $F(x_1, x_2,\dots x_s)$, $(s\geq 4)$ darstellbar sind., Izv. Akad. Nauk SSSR, 7, 111-122, 165-196, 1929.
\bibitem{Waldspurger1} J.-L. Waldspurger, Sur les coefficients de Fourier des formes modulaires de poids demi-entier J. Math Pures et Appl., 60, 375-484, 1981.


\end{thebibliography}
\end{document}